\documentclass[10pt,a4paper]{amsart}
\usepackage{amssymb}

\usepackage{fourier}
\usepackage{ulem}
\usepackage{amsthm}
\usepackage{mathrsfs}
\usepackage{graphicx}
\usepackage{stmaryrd}
\usepackage{float}
\usepackage{amsmath}
\usepackage[all]{xy}
\usepackage{xypic}
\usepackage{hyperref}

\hypersetup{
  pdftitle   = {},
  pdfauthor  = {},
  pdfcreator = {\LaTeX\ with package \flqq hyperref\frqq}
}
\DeclareMathAlphabet{\mathpzc}{OT1}{pzc}{m}{it}
\newtheorem{theorem}{Theorem}[section]
\newtheorem*{theorem*}{Theorem}
\newtheorem{proposition}[theorem]{Proposition}
\newtheorem{lemma}[theorem]{Lemma}
\newtheorem*{lemma*}{Lemma}
\newtheorem{corollary}[theorem]{Corollary}

\newtheorem*{conjecture*}{Conjecture}

\theoremstyle{definition}
\newtheorem{definition}[theorem]{Definition}
\newtheorem{example}[theorem]{Example}
\newtheorem{notation}[theorem]{Notation}
\theoremstyle{remark}
\newtheorem{remark}[theorem]{Remark}

\DeclareMathOperator{\grad}{grad}

\numberwithin{equation}{section}

\begin{document}
\title{ Symplectic Lefschetz fibrations on adjoint orbits}
\author{Elizabeth Gasparim, Lino Grama, and Luiz A. B. San Martin}
\address{Gasparim - Departamento de Matem\'aticas, Universidad Cat\'olica del Norte, Antofagasta, Chile\qquad \qquad \phantom{xxx}
  Grama and San Martin - Imecc -
Unicamp, Departamento de Matem\'{a}tica. Rua S\'{e}rgio Buarque de Holanda,
651, Cidade Universit\'{a}ria Zeferino Vaz. 13083-859 Campinas - SP, Brasil.\qquad \qquad \phantom{xxxxxxxxxxx}
E-mails: smartin@ime.unicamp.br, etgasparim@gmail.com, linograma@gmail.com.}

\begin{abstract}
We prove that adjoint orbits of semisimple Lie algebras have the structure
of symplectic Lefschetz fibrations. 
 We  describe the topology of the regular and singular
fibres, in particular we calculate their middle Betti numbers.
\end{abstract}

\thanks{The authors acknowledge support of Fapesp grant 2012/10179-5, Fapesp
grant 2012/07482-8 , Fapesp grant 2012/18780-0 and CNPq grant 303755/2009-1.}
\maketitle
\tableofcontents

\section{Motivation and statements of  results}

%
%

Our goal in this work is to construct  symplectic Lefschetz fibrations in  dimensions higher than 4. 
Somewhat surprinsingly,  we found effective tools to
 construct such fibrations in Lie theory.
 
The  literature about SLFs in 4 real dimensions is vast.
A celebrated result of Donaldson \cite{Do}
shows that  after blowing up finitely many points, every symplectic manifold 4-manifold
admits a Lefschetz fibration.
Conversely,  the existence of a topological
Lefschetz fibration on a  4 dimensional symplectic manifold guaranties the existence of an
SLF  whenever the fibres have genus at least 2, see \cite{GoS}.
Moreover, 
Amor\'os--Bogomolov--Katzarkov--Pantev proved existence  SLFs  in 4D with arbitrary fundamental group
\cite{ABKP}.
 In general, it is possible to construct Lefschetz fibrations starting up
with a Lefschetz pencil and then blowing up its base locus (see \cite{Se}, 
\cite{Go}). However, in such cases one needs to fix the indefiniteness of
the symplectic form over the exceptional locus by glueing in a correction,
and this makes it rather difficult to explicitly find Lagrangean vanishing cycles.
 Direct constructions of Lefschetz fibrations in higher dimensions
were by and large lacking in the literature.

 Our construction does not make use of
Lefschetz pencils, we construct our symplectic Lefschetz fibrations directly
using  height functions that come naturally
from Lie theory.
We prove that adjoint orbits of semisimple Lie algebras have the structure
of symplectic Lefschetz fibrations. We then describe the topology of the
fibres, in particular calculating their middle Betti numbers.  Our main results are: \vspace{3mm}

\noindent \textbf{Theorem \ref{teofibracao}} Let $\mathfrak{h}$ be the
Cartan subalgebra of a complex semisimple Lie algebra. Given $H_{0}\in 
\mathfrak{h}$ and $H\in \mathfrak{h}_{\mathbb{R}}$ with $H$ a regular
element. The \textit{height function} \ $f_{H}\colon\mathcal{O}\left(
H_{0}\right) \rightarrow \mathbb{C}$ defined by 
\begin{equation*}
f_{H}\left( x\right) =\langle H,x\rangle \qquad x\in \mathcal{O}\left(
H_{0}\right) 
\end{equation*}%
has a finite number (= $|\mathcal{W}|/|\mathcal{W}_{H_{0}}|$) of isolated
singularities and gives $\mathcal{O}\left(H_{0}\right) $ the structure of a
symplectic Lefschetz fibration. \vspace{3mm}

The precise meaning of this statement is explained in section \ref{LF}, and
%
In section \ref{reg} we describe the topology of the regular fibre, and in
section \ref{sing} we describe the singular fibre, obtaining: \vspace{3mm}

\noindent \textbf{Corollary \ref{cor.homology}} The homology of a regular
level $L\left( \xi \right) $ coincides with that of  $\mathbb{F}_{H_0}\setminus \mathcal{W}\cdot H_{0}$.  In particular, the middle  Betti number
of $L\left( \xi \right)$ equals $k-1$, where $k$ is the number of
singularities of the  fibration $f_H$ (and equals the number of elements in
the orbit $\mathcal{W}\cdot H_{0}$).

\vspace{3mm}

\noindent \textbf{Corollary \ref{cor.sing}}
The homology of a singular level $L\left( w H_0\right) $, $w \in \mathcal{W}$ 
coincides with that of \,  
$$\mathbb{F}_{H_0} \setminus  \{uH_0 \in \mathcal{W}\cdot H_{0}: u\neq w\}\text{.}$$
 In particular, the middle  Betti number of $L\left( wH_0 \right)$
equals $k-2$, where $k$ is the number of singularities of the  fibration $f_H.$

%
%
%

\vspace{3mm}

\noindent\textbf{Acknowledgments} 
We thank Ron Donagi, Ludmil Katzarkov and Tony Pantev for enlightening
discussions.

\section{Lefschetz Fibrations on adjoint orbits}

Let $\mathfrak{g}$ be a complex semisimple Lie algebra and $G$ a connected
Lie group with Lie algebra $\mathfrak{g}$ (for instance $G$ could be $%
\mathrm{Aut}_{0}\left( \mathfrak{g}\right) $, the connected component of the
identity of the automorphism group of $G$).

The Cartan--Killing form of $\mathfrak{g}$, $\langle X,Y\rangle =\mathrm{tr}%
\left( \mathrm{ad}\left( X\right) \mathrm{ad}\left( Y\right) \right) \in 
\mathbb{C}$,  is symmetric and nondegenerate. Moreover,  $\langle \cdot
,\cdot \rangle $ is invariant by the adjoint representation, that is 
\begin{equation*}
\langle \lbrack X,Y],Z\rangle =-\langle Y,[X,Z]\rangle \qquad X,Y,Z\in 
\mathfrak{g}. 
\end{equation*}

Fix a Cartan subalgebra $\mathfrak{h}\subset \mathfrak{g}$ and a compact
real form  $\mathfrak{u}$ of $\mathfrak{g}$. Associated to these subalgebras
there are the subgroups $T=\langle \exp \mathfrak{h}\rangle =\exp \mathfrak{h%
}$ and $U=\langle \exp \mathfrak{u} \rangle =\exp \mathfrak{u}$. Denote by $%
\tau$ the conjugation associated to  $\mathfrak{u}$, defined by $\tau \left(
X\right) =X$ if $X\in \mathfrak{u}$ and $\tau \left( Y\right) =-Y$ if $Y\in i%
\mathfrak{u}$. Hence if $Z=X+iY\in \mathfrak{g}$ with $X,Y\in \mathfrak{u}$
then $\tau \left( X+iY\right) =X-iY$. In this case, the sesquilinear form $%
\mathcal{H}_{\tau }:\mathfrak{g}\times \mathfrak{g}\rightarrow \mathbb{C}$ 
defined by 
\begin{equation}  \label{hermitian}
\mathcal{H}_{\tau }\left( X,Y\right) =-\langle X,\tau Y\rangle
\end{equation}
is a Hermitian form on $\mathfrak{g}$ (see \cite[lemma 12.17]{amalglie}).

A root of $\mathfrak{h}$ is a linear functional $\alpha :\mathfrak{h}%
\rightarrow \mathbb{C}$, $\alpha \neq 0$, such that the space of roots 
\begin{equation*}
\mathfrak{g}_{\alpha }=\{X \in \mathfrak{g}:\forall H\in \mathfrak{h}%
,~[H,X]=\alpha \left( H\right) X\}\neq \{0\}. 
\end{equation*}%
The set of all roots is denoted by $\Pi $. The decomposition  $\mathfrak{g}$
in eigenspaces of $\mathrm{ad}\left( H\right) $, $H\in \mathfrak{h}$, is
given by 
\begin{equation*}
\mathfrak{g}=\mathfrak{h}\oplus \sum_{\alpha \in \Pi }\mathfrak{g}_{\alpha
}. 
\end{equation*}%
An element $H\in \mathfrak{h}$ is \textbf{regular} if $\alpha \left(
H\right) \neq 0$ for all $\alpha \in \Pi $.

The restriction of the Cartan--Killing form to $\mathfrak{h}$ is
nondegenerate so we can define, for each $\alpha \in \Pi $, $H_{\alpha }\in 
\mathfrak{h}$ by $\alpha \left( \cdot \right) =\langle H_{\alpha },\cdot
\rangle $. The real subspace generated by $H_{\alpha }$, $\alpha \in \Pi $,
is denoted by $\mathfrak{h}_{\mathbb{R}}$. In the canonical construction of $%
\mathfrak{u}$ we have $\mathfrak{h}_{\mathbb{R}}\subset i\mathfrak{u}$.

The Weyl group $\mathcal{W}$ is given by $\mathcal{W}=\mathrm{Nor}_{G}\left( 
\mathfrak{h}\right) /\mathrm{Cent}_{G}\left( \mathfrak{h}\right) $
(normaliser modulo centraliser) or, equivalently, the group generated by
reflexions with respect to the roots. $\mathcal{W}$ is finite.

The adjoint representation of $G$ in $\mathfrak{g}$ is denoted by $\mathrm{Ad%
}\left( g\right) X$, $g\in G$ and $X\in \mathfrak{g}$, or simply by $g\cdot X
$. An adjoint orbit is given by 
\begin{equation*}
\mathcal{O}\left( X\right) =G\cdot X=\{g\cdot X\in \mathfrak{g}:g\in G\}. 
\end{equation*}%
Such an orbit can be identified with the quotient space $G/\mathrm{Cent}%
_{G}\left( X\right) $ where $\mathrm{Cent}_{G}\left( X\right) =\{g\in
G:g\cdot X=X\}$ is the \textit{centraliser} of $X$ in $G$. If $H\in 
\mathfrak{h}$ is regular then $\mathrm{Cent}_{G}\left( H\right) =T=\exp 
\mathfrak{h}$. The tangent space $T_{x}\mathcal{O}\left( X\right) $ to the
orbit $\mathcal{O}\left( X\right) $ at $x$ is given by 
\begin{eqnarray*}
T_{x}\mathcal{O}\left( X\right) &=&\Im\mathrm{ad}\left( x\right)
=\{[x,A]:A\in \mathfrak{g}\} \\
&=&\{[A,x]:A\in \mathfrak{g}\}
\end{eqnarray*}%
since $[A,x]=\frac{d}{dt}\left( e^{t\mathrm{ad}\left( A\right) }x\right)
_{\left\vert t=0\right. }$ and $e^{t\mathrm{ad}\left( A\right) }=\mathrm{Ad}%
\left( e^{tA}\right) $.

Note that, because  $\mathfrak{g}$ is a complex Lie algebra, the tangent
spaces $T_{x}\mathcal{O}\left( X\right) $ to $\mathcal{O}\left( X\right) $
are complex subspaces of $\mathfrak{g}$, since if $[A,x]$ is a tangent
vector then $i[A,x]=[iA,x]$ is also a tangent vector. This implies that each
adjoint orbit  $\mathcal{O}\left( X\right) $ is a complex manifold, as it is
endowed with an almost complex structure  (multiplication by $i$ in each
tangent space) which is integrable, simply because this almost complex
structure is the restriction of a complex structure on $\mathfrak{g}$ (the
Nijenhuis tensor vanishes).

\vspace{12pt}%

\noindent%

\begin{example}
\label{ex-sl2}When $\mathfrak{g}=\mathfrak{sl}\left( n,\mathbb{C}\right) $
the data just described is:

\begin{enumerate}
\item $\langle \cdot ,\cdot \rangle $ is a (constant) multiple of the form $%
\mathrm{tr}\left( XY\right) $;

\item A canonical choice of $\mathfrak{h}$ is the subalgebra of diagonal
matrices;

\item with this choice of $\mathfrak{h}$ the roots are the linear
functionals  $\alpha _{ij}\left( \mathrm{diag}\{a_{1},\ldots ,a_{n}\}\right)
=a_{i}-a_{j}$, $i\neq j$, with $\mathfrak{g}_{\alpha _{ij}}$ the subspace
generated by the basis element given by the matrix $E_{ij}$ (with $1$ in the 
$i,j$ entry and zeros elsewhere);

\item $\mathfrak{u}=\mathfrak{su}\left( n\right) $, the (real) algebra of
anti-Hermitian matrices. In this case $\tau \left( Z\right) =-\overline{Z}%
^{T}$, $Z\in \mathfrak{sl}\left( n,\mathbb{C}\right) $ and the associated
Hermitian form is a multiple of $\mathcal{H}_{\tau }\left( X,Y\right) =%
\mathrm{tr}\left( X\overline{Y}^{T}\right) $;

\item $H\in \mathfrak{h}$ is regular if and only if its eigenvalues are all
distinct;

\item $\mathcal{W}$ is the permutation group of $n$ elements, which acts
upon $\mathfrak{h}$ by permuting its diagonal entries.

\item If $H\in \mathfrak{h}$ then $\mathcal{O}\left( H\right) $ is the set
of diagonalizable matrices that have the same eigenvalues as $H$.
\end{enumerate}
\end{example}

\subsection{Main Theorem}
\label{LF}

The Lefschetz fibration on an adjoint orbit is the following:

\begin{theorem}
\label{teofibracao}Given $H_{0}\in \mathfrak{h}$ and $H\in \mathfrak{h}_{%
\mathbb{R}}$ with $H$ a regular element. Then, the \textquotedblleft height
function \textquotedblright\ $f_{H}:\mathcal{O}\left( H_{0}\right)
\rightarrow \mathbb{C}$ defined by 
\begin{equation*}
f_{H}\left( x\right) =\langle H,x\rangle \qquad x\in \mathcal{O}\left(
H_{0}\right) 
\end{equation*}%
has a finite number (= $|\mathcal{W}|/|\mathcal{W}_{H_{0}}|$) of isolated
singularities and defines a symplectic Lefschetz fibration, thus the
following properties hold:

\begin{enumerate}
\item The singularities are nondegenerate (Hessian non degenerate).

\item If $c_{1},c_{2}\in \mathbb{C}$ are regular values then the level
manifolds $f_{H}^{-1}\left( c_{1}\right) $ and $f_{H}^{-1}\left(
c_{2}\right) $ are diffeomorphic.

\item There exists a symplectic form $\Omega $ in $\mathcal{O}\left(
H_{0}\right) $ such that if $c\in \mathbb{C}$ is a regular value then the
level manifold $f_{H}^{-1}\left( c\right) $ is symplectic, that is, the
restriction of \, $\Omega $ to $f_{H}^{-1}\left( c\right) $ is a symplectic
(nondegenerate) form.

\item If $c\in \mathbb{C}$ is a singular value, then $f_{H}^{-1}\left(
c\right) $ is a union of  affine subspaces (contained in $\mathcal{O}\left(
H_{0}\right) $). These subspaces are symplectic with respect to the form \, $%
\Omega $ from the previous item.
\end{enumerate}
\end{theorem}

The proof will be carried out in several steps.

\begin{remark}
\label{potential}

The height function $f_{H}$ defined by an element $H\in \mathfrak{h}_{%
\mathbb{R}}$ is extensively used in the study of the geometry of flag
manifolds. This is due to the fact that it is a Morse--Bott function in
general, which is Morse if $H$ is regular. These height functions make the
link between Morse theory and the algebraic theory of Bruhat decompositions.
This is because the gradient $\grad f_{H}$ of $f_{H}$, with respect to the
so called Borel metric is precisely the vector field $\widetilde{H}$ induced
by $H$ on a flag manifold (see Duistermaat--Kolk--Varadarajan \cite{dkv}).
The unstable manifolds of $\grad f_{H}=\widetilde{H}$ are the components of
the Bruhat decomposition if $H$ is regular. For applications of these height
functions to the geometry of flag manifolds see Kocherlakota \cite{koch},
regarding the Morse homology, and the extensive literature on the
\textquotedblleft convexity theorems\textquotedblright\ started with Kostant 
\cite{kos}, Atiyah \cite{at} and Guillemin--Sternberg \cite{gs}.
\end{remark}

\subsection{Singularities of the fibration}

First of all, if $A\in \mathfrak{g}$ and $x\in \mathcal{O}\left(
H_{0}\right) $ then $[A,x]$ is a vector tangent to $\mathcal{O}\left(
H_{0}\right) $ at $x$ and the differential of $f_{H}$ is given by 
\begin{equation}
\left( df_{H}\right) _{x}\left( [A,x]\right) =\frac{d}{dt}\langle H,e^{t%
\mathrm{ad}\left( A\right) }x\rangle _{\left\vert t=0\right. }=\langle
H,[A,x]\rangle =\langle \lbrack x,H],A\rangle .  \label{forderivfaga}
\end{equation}%
From this expression it follows that  $f_{H}$ is a holomorphic function with
respect to the complex structure of $\mathcal{O}\left( H_{0}\right) $.
Indeed, 
\begin{equation*}
\left( df_{H}\right) _{x}\left( i[A,x]\right) =\left( df_{H}\right)
_{x}\left( [iA,x]\right) =\langle \lbrack x,H],iA\rangle =i\langle \lbrack
x,H],A\rangle =i\left( df_{H}\right) _{x}\left( [A,x]\right) . 
\end{equation*}%
Being a holomorphic function, the rank of $f_{H}$ at $x\in \mathcal{O}\left(
H_{0}\right) $ (regarded as a map taking values in $\mathbb{R}^{2}\approx 
\mathbb{C}$) is either $0$ or $2$, given that if $\left( df_{H}\right)
_{x}\left( [A,x]\right) \neq 0$ then $i\left( df_{H}\right) _{x}\left(
[A,x]\right) \neq 0$ and these two derivatives generate $\mathbb{R}%
^{2}\approx \mathbb{C}$. In particular, this means that  $x\in \mathcal{O}%
\left( H_{0}\right) $ is a singular point of $f_{H}$ if and only if $\left(
df_{H}\right) _{x}=0$.

Therefore, by expression (\ref{forderivfaga}) for the differential of $f_{H}$%
, it follows that $x$ is a singularity, that is, $\left( df_{H}\right)
_{x}\left( [A,x]\right) =0$ for all $A\in \mathfrak{g}$ if and only if  $%
[x,H]=0$. This allows us to identify the singular points.

\begin{proposition}
\label{singularities} $x$ is a singular point for $f_{H}$ if and only if $%
x\in \mathcal{O}\left( H_{0}\right) \cap \mathfrak{h}=\mathcal{W}\cdot H_{0}$%
, where $\mathcal{W}$ is the Weyl group. (At this point the hypothesis that $%
H$ is regular is used.)
\end{proposition}

\begin{proof}
As observed, $x$ is a singularity if and only if $[x,H]=0$. But, as $H$ is
regular its centralizer is the  Cartan subalgebra $\mathfrak{h}$ itself. It
follows that the singularity set is  $\mathcal{O}\left( H_{0}\right) \cap 
\mathfrak{h}$. This set is exactly the orbit of $H_{0}$ by the action of $%
\mathcal{W}$.
\end{proof}

Since $\mathcal{W}$ is finite we obtain the following corollary.

\begin{corollary}
The set of singularities of $f_{H}$ is finite.
\end{corollary}

To obtain the Hessian at a singularity $x_{0}\in \mathcal{O}\left(
H_{0}\right) \cap \mathfrak{h}$, take $B\in \mathfrak{g}$. Then the second
derivative at $x\in \mathcal{O}\left( H_{0}\right) $ calculated at  $[A,x] $
and $[B,x]$ is given by 
\begin{eqnarray*}
\frac{d}{dt}\langle \lbrack e^{t\mathrm{ad}\left( B\right) }x,H],A\rangle
_{\left\vert t=0\right. } &=&\langle \lbrack B,x],H],A\rangle \\
&=&\langle \lbrack \lbrack B,H],x],A\rangle +\langle \lbrack
B,[x,H]],A\rangle .
\end{eqnarray*}%
In particular, if $x_{0}$ is a singularity then $[x_{0},H]=0$ and the second
derivative becomes 
\begin{equation}
\langle \lbrack \lbrack B,H],x_{0}],A\rangle =\langle \lbrack
x_{0},[H,B]],A\rangle .  \label{forhessiano}
\end{equation}

\begin{proposition}
The second term of (\ref{forhessiano}) defines a symmetric bilinear form
whose restriction to the tangent space $T_{x_{0}}\mathcal{O}\left(
H_{0}\right) $ at $x_{0}\in \mathfrak{h}$ is nondegenerate.
\end{proposition}

\begin{proof}
The tangent space $T_{x_{0}}\mathcal{O}\left( H_{0}\right) $ is the image of 
$\mathrm{ad}\left( x_{0}\right) $, which equals 
\begin{equation*}
\mathrm{im}\left( \mathrm{ad}\left( x_{0}\right) \right) =\sum_{\alpha
\left( x_{0}\right) \neq 0}\mathfrak{g}_{\alpha } 
\end{equation*}%
given that $\mathrm{ad}\left( x_{0}\right) $ is diagonalizable and its
eigenvalues are $0$ and $\alpha \left( x_{0}\right) $, $\alpha \in \Pi $.
From this we observe that the restriction of  $\mathrm{ad}\left(
x_{0}\right) $ to its image is an invertible linear map.  Therefore, the
tangent vectors $[x_{0},A]$ with $A$ varying inside  $\mathrm{im}\left( 
\mathrm{ad}\left( x_{0}\right) \right) $ cover the entire tangent space $%
T_{x_{0}}\mathcal{O}\left( H_{0}\right) $. This means that in the second
derivative (\ref{forhessiano}) we can restrict $A$ and $B$ to $\mathrm{im}%
\left( \mathrm{ad}\left( x_{0}\right) \right) $.

Now, on one hand the restriction of  $\mathrm{ad}\left( H\right) $ to $%
\mathrm{im}\left( \mathrm{ad}\left( x_{0}\right) \right) $ is also
invertible since $H$ is regular. On the other hand, the restriction of the
Cartan--Killing form to  $\mathrm{im}\left( \mathrm{ad}\left( x_{0}\right)
\right) $ is nondegenerate, since if $\alpha \left( x_{0}\right) \neq 0$
then $\left( -\alpha \right) \left( x_{0}\right) \neq 0$ and given $Y\in 
\mathfrak{g}_{\alpha }$ there exists $Z\in \mathfrak{g}_{-\alpha }$ such
that $\langle Y,Z\rangle \neq 0$.

The upshot is that the expression $\langle \lbrack x_{0},[H,B]],A\rangle $
with $A,B\in \mathrm{im}\left( \mathrm{ad}\left( x_{0}\right) \right) $
takes the form $\mathfrak{B}\left( Pu,v\right) $ where $\mathfrak{B}$ is a
nondegenerate bilinear form and $P$ is an invertible linear transformation
on a vector space. Such a bilinear form is always nondegenerate.
\end{proof}

This proposition concludes the proof of item  (1) of theorem \ref%
{teofibracao}.

\subsection{Diffeomorphisms among regular fibres}

To show that the inverse images of two regular points are diffeomorphic, we
construct vector fields transversal to the fibres in such a way that for a
given fibre the flows of these vectors fields are well defined up to a
certain time in all the fibre (as $\mathcal{O}\left( H_{0}\right) $ is not
compact, it is not to be expected that the vector fields be complete). The
diffeomorphism is obtained form such flows.

The transversal vector fields that will play the appropriate roles are
defined by 
\begin{equation}
Z\left( x\right) =\frac{1}{\left\Vert [x,H]\right\Vert ^{2}}[x,[\tau x,H]]
\label{forcampoz}
\end{equation}%
where $\tau :\mathfrak{g}\rightarrow \mathfrak{g}$ is conjugation with
respect to the real compact form $\mathfrak{u}$ and $\left\Vert \cdot
\right\Vert $ is the norm associated to the Hermitian form  $\mathcal{H}$.
Here are a few observations about this vector field:

\begin{enumerate}
\item $Z$ is well defined if $[x,H]\neq 0$, that is, if $x\notin \mathfrak{h}
$. Therefore, $Z$ can be regarded as a vector field on $\mathfrak{g}%
\setminus \mathfrak{h}$, which restricts to a vector field on the set of
regular points of $\mathcal{O}\left( H_{0}\right) \setminus \mathfrak{h}$.

\item If $x\in \mathcal{O}\left( H_{0}\right) \setminus \mathfrak{h}$ then  $%
Z\left( x\right) $ is tangent to $\mathcal{O}\left( H_{0}\right) $ since $%
[x,[\tau x,H]]\in \mathrm{im}\left( \mathrm{ad}\left( x\right) \right) $ is
tangent to $\mathcal{O}\left( H_{0}\right) $ at $x$. Therefore, $Z $ does
indeed restrict to a vector field in $\mathcal{O}\left( H_{0}\right)
\setminus \mathfrak{h}$.

\item Since, by hypothesis, for $H\in \mathfrak{h}_{\mathbb{R}}$, $\tau H=-H$
it follows that $[\tau x,H]=-[\tau x,\tau H]=-\tau \lbrack x,H]$.

\item The differential of $f_{H}$ at $x\in \mathcal{O}\left( H_{0}\right)
\setminus \mathfrak{h}$ satisfies 
\begin{eqnarray*}
\left( df_{H}\right) _{x}\left( [x,[\tau x,H]]\right) &=&-\langle H,[x,[\tau
x,H]]\rangle =\langle H,[x,\tau \lbrack x,H]]\rangle \\
&=&-\langle \lbrack x,H],\tau \lbrack x,H]]\rangle \\
&=&\mathcal{H}\left( [x,H],[x,H]\right) =\left\Vert [x,H]\right\Vert ^{2}
\end{eqnarray*}%
which is $>0$ if $[x,H]\neq 0$. Therefore, $df_{H}\left( Z\left( x\right)
\right) =1$. This guarantees that $Z$ is transversal to the level surfaces
of $f_{H}$.

\item The vector field $iZ$ is also transversal. This happens because the
tangent spaces to a level surface $f_{H}^{-1}\left( c\right) $, for a
regular value $c\in \mathbb{C}$, are complex subspaces of $\mathfrak{g}$.
Therefore if $Z\left( x\right) \notin T_{x}f_{H}^{-1}\left( c\right) $ then $%
iZ\left( x\right) \notin T_{x}f_{H}^{-1}\left( c\right) $.
\end{enumerate}

\begin{lemma}
\label{lemestimdif} Let $Z:\mathfrak{g}\setminus \mathfrak{h}\rightarrow 
\mathfrak{g}$ be defined by 
\begin{equation*}
Z\left( x\right) =\frac{1}{\left\Vert [x,H]\right\Vert ^{2}}[x,[\tau x,H]] 
\end{equation*}%
where $\left\Vert \cdot \right\Vert $ is the norm corresponding to the
Hermitian form $\mathcal{H}\left( \cdot ,\cdot \right) $. Then, there exists 
$M>0$ such that for all $x\in \mathfrak{g}\setminus \mathfrak{h}$ the
following inequality holds 
\begin{equation*}
\left\Vert dZ_{x}\right\Vert \leq 2M\left( \left\Vert \mathrm{ad}\left(
H\right) \right\Vert +M\left\Vert H\right\Vert \right) \frac{\left\Vert
x\right\Vert }{\left\Vert [x,H]\right\Vert ^{2}}. 
\end{equation*}%
The constant $M>0$ depends only on the bracket of $\mathfrak{g}$.
\end{lemma}

\begin{proof}
It suffices to show that the differential of $Z$, $dZ_{x}$ is bounded as a
function of $x$. If $v\in \mathfrak{g}$ then 
\begin{equation*}
dZ_{x}\left( v\right) =-\frac{2\Re \mathcal{H}\left( [v,H],[x,H]\right) }{%
\left\Vert [x,H]\right\Vert ^{4}}[x,[\tau x,H]]+\frac{1}{\left\Vert
[x,H]\right\Vert ^{2}}\left( [v,[\tau x,H]]+[x,[\tau v,H]]\right) . 
\end{equation*}%
To estimate $\left\Vert dZ_{x}\left( v\right) \right\Vert $ (and thus also $%
\left\Vert dZ_{x}\right\Vert $) we use the following inequalities:

\begin{enumerate}
\item $|\Re \mathcal{H}\left( [v,H],[x,H]\right) |\leq |\mathcal{H}\left(
[v,H],[x,H]\right) |\leq \left\Vert \lbrack x,H]\right\Vert \cdot \left\Vert 
\mathrm{ad}\left( H\right) \right\Vert \cdot \left\Vert v\right\Vert $, by
the Cauchy--Schwarz inequality, where $\left\Vert \mathrm{ad}\left( H\right)
\right\Vert $ is the operator norm of $\mathrm{ad}\left( H\right) $.

\item The bracket of a finite dimensional Lie algebra is a continuous
bilinear map, hence there exists $M>0$ such that for all  $X,Y\in \mathfrak{g%
}$ we have $\left\Vert [X,Y]\right\Vert \leq M\left\Vert X\right\Vert \cdot
\left\Vert Y\right\Vert $. Consequently,

\begin{enumerate}
\item $\left\Vert [x,[\tau x,H]]\right\Vert \leq M\left\Vert [\tau
x,H]\right\Vert \cdot \left\Vert x\right\Vert $. Since $\tau $ is an
isometry of the Hermitian form $\mathcal{H}$ and $H\in \mathfrak{h}_{\mathbb{%
R}} $, $\left\Vert [\tau x,H]\right\Vert =\left\Vert -\tau \lbrack
x,H]\right\Vert =\left\Vert [x,H]\right\Vert $. Therefore, the second term
of this inequality equals $M\left\Vert [x,H]\right\Vert \cdot \left\Vert
x\right\Vert $.

\item $\left\Vert [v,[\tau x,H]]\right\Vert $ e $\left\Vert [x,[\tau
v,H]]\right\Vert $ are bounded above by $M^{2}\left\Vert H\right\Vert \cdot
\left\Vert x\right\Vert \cdot \left\Vert v\right\Vert $.
\end{enumerate}
\end{enumerate}

An application of the triangle inequality to $\left\Vert dZ_{x}\left(
v\right) \right\Vert $, combined with the previous expression, gives us 
\begin{equation*}
\left\Vert dZ_{x}\left( v\right) \right\Vert \leq 2\left( \frac{M\left\Vert 
\mathrm{ad}\left( H\right) \right\Vert \cdot \left\Vert x\right\Vert }{%
\left\Vert [x,H]\right\Vert ^{2}}+\frac{M^{2}\left\Vert H\right\Vert \cdot
\left\Vert x\right\Vert }{\left\Vert [x,H]\right\Vert ^{2}}\right)
\left\Vert v\right\Vert , 
\end{equation*}%
from which the claimed inequality follows.
\end{proof}

Now we find estimates for $\frac{\left\Vert x\right\Vert }{\left\Vert
[x,H]\right\Vert ^{2}}$ over open subsets of $\mathcal{O}\left( H_{0}\right) 
$ which will allow us to show that, over these open sets, $\left\Vert
dZ_{x}\right\Vert $ is bounded and, consequently, that $Z$ is Lipschitz.

\begin{lemma}
There exists $C>0$ such that if $x\in \mathcal{O}\left( H_{0}\right) $ then $%
\left\Vert x\right\Vert >C$.
\end{lemma}

\begin{proof}
The point is that in a semisimple Lie algebra an adjoint orbit $\mathcal{O}%
\left( X\right) $ is closed if $\mathrm{ad}\left( X\right) $ is
diagonalizable. In particular, $\mathcal{O}\left( H_{0}\right) $ is closed
and does not contain the origin. Therefore, $\mathcal{O}\left( H_{0}\right) $
does not approach $0$ and it follows that $\inf_{x\in \mathcal{O}\left(
H_{0}\right) }\left\Vert x\right\Vert >0$.
\end{proof}

The following lemma from linear algebra will be used to estimate $\left\Vert
dZ_{x}\right\Vert $.

\begin{lemma}
\label{lemmalglin} Let $D_{n}$ and $X_{n}$ be sequences of complex matrices
such that

\begin{enumerate}
\item Each $D_{n}$ is diagonalizable and $\lim D_{n}=\infty $.

\item $\lim X_{n}=0$.
\end{enumerate}

Then there exists a subsequence  $n_{k}$ with $\lambda _{n_{k}}\in \mathbb{C}
$ such that $\lim_{k}\lambda _{n_{k}}=\infty $ e $\lambda _{n_{k}}$ is an
eigenvalue of $M_{n_{k}}=D_{n_{k}}+X_{n_{k}}$.
\end{lemma}

\begin{proof}
Denote by $a_{n}$ the diagonal entry of $D_{n}$ that has the largest
absolute value among all diagonal entries of  $D_{n}$. Then $\lim
a_{n}=\infty $, since $\lim D_{n}=\infty $. Consider the sequence 
\begin{equation*}
M_{n}=\frac{1}{a_{n}}\left( D_{n}+X_{n}\right) . 
\end{equation*}%
We have $\lim \frac{1}{a_{n}}X_{n}=0$. On the other hand, $\frac{1}{a_{n}}%
D_{n}$ is a bounded sequence, therefore there exists a subsequence $n_{k} $
such that $\lim_{k}\frac{1}{a_{n_{k}}}D_{n_{k}}=D$. Consequently, $\lim_{k}%
\frac{1}{a_{n_{k}}}M_{n_{k}}=D$. We may refine the subsequence  $n_{k}$ such
that the entry $a_{n_{k}}$ of $D_{n_{k}} $ occurs always at the same
position for all $k$. Thus $D$ is a diagonal matrix with $1$ as an
eigenvalue, since there exists a diagonal entry such that for all $k$, the
entry of $\frac{1}{a_{n_{k}}}D_{n_{k}}$ in this position is $1$.

The limit $\lim_{k}\frac{1}{a_{n_{k}}}M_{n_{k}}=D$ guarantees that for all $%
\varepsilon >0$ there exists $k_{0}\in \mathbb{N}$ such that if $k\geq k_{0}$
then  $\frac{1}{a_{n_{k}}}M_{n_{k}}$ has an eigenvalue $\mu _{n_{k}}$ with $%
|\mu _{n_{k}}-1|<\varepsilon $. Setting $\varepsilon =1/2$ we obtain $|\mu
_{n_{k}}|>1/2$. Therefore, $\lambda _{n_{k}}=a_{n_{k}}\mu _{n_{k}}$ is an
eigenvalue of $M_{n_{k}}$ and $\lim \lambda _{n_{k}}=\infty $.
\end{proof}

The following lemma shows that the adjoint orbit $\mathcal{O}\left(
H_{0}\right) $ is not asymptotic to the Cartan subalgebra $\mathfrak{h}$.

\begin{lemma}
\label{lemepsdelta} Let $\mathcal{O}\left( H_{0}\right) \cap \mathfrak{h}$
be the finite set of singularities of $f_{H}$ in $\mathcal{O}\left(
H_{0}\right) $. Given $\varepsilon >0$ denote by $O_{\varepsilon }$ the set
of $x\in \mathcal{O}\left( H_{0}\right) $ which are at a distance greater
than $\varepsilon $ of the singularities:%
\begin{equation*}
O_{\varepsilon }=\{x\in \mathcal{O}\left( H_{0}\right) :\forall y\in 
\mathcal{O}\left( H_{0}\right) \cap \mathfrak{h},~\left\Vert x-y\right\Vert
>\varepsilon \}. 
\end{equation*}%
Denote by $p:\mathfrak{g}\rightarrow \sum_{\alpha \in \Pi }\mathfrak{g}%
_{\alpha }$ the projection given by the decomposition $\mathfrak{g}=%
\mathfrak{h}\oplus \sum_{\alpha \in \Pi }\mathfrak{g}_{\alpha }$. Then we
have the following properties:

\begin{enumerate}
\item Given $\varepsilon >0$ there exists $\delta >0$ such that, if $x\in
O_{\varepsilon }$, then $\left\Vert p\left( x\right) \right\Vert >\delta $.

\item There exists a constant $\Gamma _{\varepsilon }>0$ such that if $x\in
O_{\varepsilon }$ then%
\begin{equation*}
\frac{\left\Vert x-p\left( x\right) \right\Vert }{\left\Vert p\left(
x\right) \right\Vert }<\Gamma _{\varepsilon }. 
\end{equation*}
\end{enumerate}
\end{lemma}

\begin{proof}
Both properties are proved by contradiction.

\begin{enumerate}
\item Assume the statement is false. Then there exist $\varepsilon >0$ and a
sequence $y_{n}\in O_{\varepsilon }$ such that $\lim_{n}p\left( y_{n}\right)
=0$. Set $y_{n}=H_{n}+Y_{n}$, with $H_{n}\in \mathfrak{h}$ and $%
Y_{n}=p\left( y_{n}\right) $. The contradiction hypothesis guarantees that $%
\lim y_{n}=\infty $, since otherwise there would exist a subsequence  $%
y_{n_{k}}$ with $\lim_{k}y_{n_{k}}=y$. This implies that $\lim H_{n_{k}}=y$
given that $\lim Y_{n_{k}}=0$. Since $\mathfrak{h}$ and $\mathcal{O}\left(
H_{0}\right) $ are closed, it follows that $y\in \mathcal{O}\left(
H_{0}\right) \cap \mathfrak{h}$, contradicting the fact that $y_{n}$ does
not approach $\mathcal{O}\left( H_{0}\right) \cap \mathfrak{h}$.
Consequently, $\lim H_{n}=\infty $.

We may now apply lemma \ref{lemmalglin} by taking $D_{n}=\mathrm{ad}\left(
H_{n}\right) $ and $X_{n}=\mathrm{ad}\left( Y_{n}\right) $. This shows that
there exists a subsequence $n_{k}$ such that $\mathrm{ad}\left(
y_{n_{k}}\right) =D_{n_{k}}+X_{n_{k}}$ has an eigenvalue $\lambda _{n_{k}}$
with $\lim \lambda _{n_{k}}=\infty $. But this is a contradiction because $%
y_{n}\in \mathcal{O}\left( H_{0}\right) $ and, therefore, the eigenvalues of 
$\mathrm{ad}\left( y_{n}\right) $ are the same as the eigenvalues of $%
\mathrm{ad}\left( H_{0}\right) $.

\item Assume the statement is false. Then there exists a sequence $y_{n}\in
O_{\varepsilon }$ such that $\lim \frac{\left\Vert y_{n}-p\left(
y_{n}\right) \right\Vert }{\left\Vert p\left( y_{n}\right) \right\Vert }%
=\infty $. That is, $\lim \frac{\left\Vert p\left( y_{n}\right) \right\Vert 
}{\left\Vert y_{n}-p\left( y_{n}\right) \right\Vert }=0$ or alternatively%
\begin{equation*}
\lim \frac{p\left( y_{n}\right) }{\left\Vert y_{n}-p\left( y_{n}\right)
\right\Vert }=0. 
\end{equation*}%
Set $H_{n}=y_{n}-p\left( y_{n}\right) \in \mathfrak{h}$, $D_{n}=\mathrm{ad}%
\left( H_{n}\right) $ and $X_{n}=\mathrm{ad}\left( p\left( y_{n}\right)
\right) $. As in the proof of lemma \ref{lemmalglin}, let $a_{n} $ be the
eigenvalue of $D_{n}$ with largest absolute value, so that $\left\Vert
D_{n}\right\Vert =|a_{n}|$. Since the adjoint map $\mathrm{ad}:\mathfrak{g}%
\rightarrow \mathfrak{gl}\left( \mathfrak{g}\right) $ is injective, there
exist constants $C_{1},C_{2}>0$ such that for all $Z\in \mathfrak{g}$ we
have $C_{1}\left\Vert \mathrm{ad}\left( Z\right) \right\Vert \geq \left\Vert
Z\right\Vert \geq C_{2}\left\Vert \mathrm{ad}\left( Z\right) \right\Vert $.
In particular, $\left\Vert H_{n}\right\Vert \geq C_{2}\left\Vert
D_{n}\right\Vert $. Therefore, 
\begin{equation*}
\lim \frac{p\left( y_{n}\right) }{|a_{n}|}=0 
\end{equation*}%
and we obtain 
\begin{equation*}
\lim \frac{X_{n}}{|a_{n}|}=0. 
\end{equation*}%
Now, to arrive at a contradiction, we proceed as in the proof of lemma  \ref%
{lemmalglin}: there exists a subsequence $n_{k}$ such that $\frac{1}{%
|a_{n_{k}}|}\left( D_{n_{k}}+X_{n_{k}}\right) $ converges to a limit which
has an eigenvalue equal to $1$. Therefore, from a certain $k_{0}$ onwards,
each $\frac{1}{|a_{n_{k}}|}\left( D_{n_{k}}+X_{n_{k}}\right) $ has an
eigenvalue with absolute value $>1/2$, which implies that $\mathrm{ad}\left(
y_{n_{k}}\right) =D_{n_{k}}+X_{n_{k}}$ has a sequence of eigenvalues that
converges to $\infty $. However, as in item (1), this is a contradiction
since $y_{n}\in \mathcal{O}\left( H_{0}\right) $ and, consequently, the
eigenvalues of $\mathrm{ad}\left( y_{n}\right)$ are the same as those of $%
\mathrm{ad}\left( H_{0}\right) $.
\end{enumerate}
\end{proof}

Now it is possible to show that $\left\Vert dZ_{x}\right\Vert $ is bounded
in $O_{\varepsilon }$ (and obviously $\left\Vert d\left( iZ\right)
_{x}\right\Vert $ is bounded as well).

\begin{lemma}
\label{lemlipschitz}Given $\varepsilon >0$ there exists $L_{\varepsilon }>0$
such that $\left\Vert dZ_{x}\right\Vert \leq L_{\varepsilon }$ if $x\in
O_{\varepsilon }$.
\end{lemma}

\begin{proof}
By lemma \ref{lemestimdif}, we have 
\begin{equation*}
\left\Vert dZ_{x}\right\Vert \leq M\left( \left\Vert \mathrm{ad}\left(
H\right) \right\Vert +M\left\Vert H\right\Vert \right) \frac{\left\Vert
x\right\Vert }{\left\Vert [x,H]\right\Vert ^{2}} 
\end{equation*}%
if $x\notin \mathfrak{h}$. In particular, this inequality holds for $x\in
O_{\varepsilon }$. Therefore, it suffices to estimate $\frac{\left\Vert
x\right\Vert }{\left\Vert [x,H]\right\Vert ^{2}}$.

Let $\delta >0$ be given as item (1) of lemma \ref{lemepsdelta}, such that $%
\left\Vert p\left( x\right) \right\Vert >\delta $ if $x\in O_{\varepsilon }$%
. Since $H$ is regular the restriction of $\mathrm{ad}\left( H\right) $ to $%
\sum_{\alpha \in \Pi }\mathfrak{g}_{\alpha }$ is an invertible linear map.
Therefore, there exists $C>0$ such that if $y\in \sum_{\alpha \in \Pi }%
\mathfrak{g}_{\alpha }$ and $\left\Vert y\right\Vert >\delta $, then  $%
\left\Vert \mathrm{ad}\left( H\right) y\right\Vert >C\left\Vert y\right\Vert 
$. This implies that if $x\in O_{\varepsilon }$, then 
\begin{equation*}
\left\Vert \lbrack H,x]\right\Vert =\left\Vert [H,H^{\prime }+p\left(
x\right) ]\right\Vert =\left\Vert [H,p\left( x\right) ]\right\Vert
>C\left\Vert p\left( x\right) \right\Vert >C\delta . 
\end{equation*}%
Consequently, choosing $\left\Vert [x,H]\right\Vert >C\delta $ as one of the
factors of the denominator and $\left\Vert [x,H]\right\Vert >C\left\Vert
p\left( x\right) \right\Vert $, it follows that 
\begin{equation*}
\frac{\left\Vert x\right\Vert }{\left\Vert [x,H]\right\Vert ^{2}}<\frac{1}{%
C^{2}\delta }\cdot \frac{\left\Vert x\right\Vert }{\left\Vert p\left(
x\right) \right\Vert }. 
\end{equation*}%
Now, $\left\Vert x\right\Vert ^{2}=\left\Vert x-p\left( x\right) \right\Vert
^{2}+\left\Vert p\left( x\right) \right\Vert ^{2}$ since $x-p\left( x\right)
\in \mathfrak{h}$ is orthogonal to $p\left( x\right) \in \sum_{\alpha \in
\Pi }\mathfrak{g}_{\alpha }$. Therefore, 
\begin{eqnarray*}
\left( \frac{\left\Vert x\right\Vert }{\left\Vert p\left( x\right)
\right\Vert }\right) ^{2} &=&\frac{\left\Vert x-p\left( x\right) \right\Vert
^{2}+\left\Vert p\left( x\right) \right\Vert ^{2}}{\left\Vert p\left(
x\right) \right\Vert ^{2}} \\
&=&\frac{\left\Vert x-p\left( x\right) \right\Vert ^{2}}{\left\Vert p\left(
x\right) \right\Vert ^{2}}+1.
\end{eqnarray*}%
By lemma \ref{lemepsdelta} (2), $\frac{\left\Vert x-p\left( x\right)
\right\Vert ^{2}}{\left\Vert p\left( x\right) \right\Vert ^{2}}<\Gamma
_{\varepsilon }^{2}$, so 
\begin{equation*}
\frac{\left\Vert x\right\Vert }{\left\Vert p\left( x\right) \right\Vert }<%
\sqrt{\Gamma _{\varepsilon }^{2}+1} 
\end{equation*}%
if $x\in O_{\varepsilon }$. This completes the proof, since 
\begin{equation*}
L_{\varepsilon }=\frac{M\left( \left\Vert \mathrm{ad}\left( H\right)
\right\Vert +M\left\Vert H\right\Vert \right) }{C^{2}\delta }\sqrt{\Gamma
_{\varepsilon }^{2}+1} 
\end{equation*}%
satisfies the desired inequality.
\end{proof}

A similar estimate shows that $Z$ is bounded in each  $O_{\varepsilon }$.

\begin{lemma}
Given $\varepsilon >0$ there exists $M_{\varepsilon }>0$ such that $%
\left\Vert Z\left( x\right) \right\Vert \leq M_{\varepsilon }$ if $x\in
O_{\varepsilon } $.
\end{lemma}

\begin{proof}
Let $M$ be as in lemma \ref{lemestimdif}. Then, 
\begin{eqnarray*}
\left\Vert Z\left( x\right) \right\Vert &=&\frac{1}{\left\Vert
[x,H]\right\Vert ^{2}}\left\Vert [x,[\tau x,H]]\right\Vert \\
&\leq &M\frac{\left\Vert x\right\Vert \cdot \left\Vert \lbrack
x,H]\right\Vert }{\left\Vert [x,H]\right\Vert ^{2}}=M\frac{\left\Vert
x\right\Vert }{\left\Vert [x,H]\right\Vert }
\end{eqnarray*}%
and, as in the proof of the previous lemma, $\frac{\left\Vert x\right\Vert }{%
\left\Vert [x,H]\right\Vert }$ in bounded on $O_{\varepsilon }$.
\end{proof}

Lemma \ref{lemlipschitz} guarantees that $Z$ is Lipschitz on $O_{\varepsilon
}$ with constant $L_{\varepsilon }$. The same is true for the vector field $%
e^{i\theta }Z$ with $\theta \in \mathbb{R}$ since $\left\Vert d\left(
e^{i\theta }Z\right) \right\Vert =\left\Vert dZ\right\Vert $. By the
previous lemma, $e^{i\theta }Z$ is bounded on $O_{\varepsilon }$. Combining
these two facts, the theory of differential equations guarantees that all
solutions of $Z$ with initial condition $x\left( 0\right) \in O_{\varepsilon
}$ extend to a common interval of definition that contains $0$.

\begin{corollary}
\label{corfluxo}Denote by $\phi _{t}^{\theta }$ the local flow of the vector
field $e^{i\theta }Z$. Then, given $\varepsilon >0$ there exists $\sigma
_{\varepsilon }>0$ such that $\phi _{t}^{\theta }\left( x\right) $ is well
defined if $t\in \left( -\sigma _{\varepsilon },\sigma _{\varepsilon
}\right) $ and $x\in O_{\varepsilon }$. Under these conditions, $\phi
_{t}^{\theta }\left( x\right) \in O_{\varepsilon }$.
\end{corollary}

We are now ready to prove item (2) of theorem \ref{teofibracao}.

\begin{proposition}
If $c_{1},c_{2}\in \mathbb{C}$ are regular values then the level manifolds $%
f_{H}^{-1}\left( c_{1}\right) $ and $f_{H}^{-1}\left( c_{2}\right) $ are
diffeomorphic.
\end{proposition}

\begin{proof}
On the set of regular values, define the equivalence relation $c_{1}\sim
c_{2}$ if $f_{H}^{-1}\left( c_{1}\right) $ and $f_{H}^{-1}\left(
c_{2}\right) $ are diffeomorphic. We must show there exists a single
equivalence class. To do so, it suffices to show that if $c\in \mathbb{C}$
is a regular value, then there exists a neighbourhood $U$ of $c$ such that
for all $d\in U$, $f_{H}^{-1}\left( d\right) $ and $f_{H}^{-1}\left(
c\right) $ are diffeomorphic. Indeed, this guarantees that the equivalence
classes are open subsets (and, consequently, closed). However, the set of
regular values is connected in $\mathbb{C}$ since it is the complement of a
finite set.

Fix a regular value $c$. Since $f_{H}^{-1}\left( c\right) $ does not
intercept the set of regular points, there exists $\varepsilon >0$ such that 
$f^{-1}\left( c\right) \subset O_{\varepsilon }$.

Let $\sigma _{\varepsilon }$ be as in corollary \ref{corfluxo}. Then  $\phi
_{t}^{\theta }\left( x\right) $ is defined for $t\in \left( -\sigma
_{\varepsilon },\sigma _{\varepsilon }\right) $ and $x\in O_{\varepsilon }$.
In particular, it is also defined for $x\in f_{H}^{-1}\left( c\right) $. For
a fixed $x $, the curve 
\begin{equation*}
\gamma _{\theta }:t\in \left( -\sigma _{\varepsilon },\sigma _{\varepsilon
}\right) \mapsto f_{H}\left( \phi _{t}^{\theta }\left( x\right) \right) \in 
\mathbb{C} 
\end{equation*}%
has derivative $\gamma _{\theta }^{\prime }\left( t\right) =\left(
df_{H}\right) _{\phi _{t}^{\theta }\left( x\right) }\left( e^{i\theta
}Z\left( \phi _{t}^{\theta }\left( x\right) \right) \right) $. However, by
definition of the field $Z$, $\left( df_{H}\right) _{y}\left( Z\left(
y\right) \right) =1$, so we have $\gamma _{\theta }^{\prime }\left( t\right)
=e^{i\theta }$. Therefore, 
\begin{eqnarray*}
\gamma _{\theta }\left( t\right) &=&\gamma _{\theta }\left( 0\right)
+\int_{0}^{t}\gamma _{\theta }^{\prime }\left( s\right) ds \\
&=&f_{H}\left( x\right) +te^{i\theta }.
\end{eqnarray*}%
That is, $f_{H}\left( \phi _{t}^{\theta }\left( x\right) \right)
=f_{H}\left( x\right) +te^{i\theta }$. In particular, if $x\in
f_{H}^{-1}\left( c\right) $ then $\phi _{t}^{\theta }\left( x\right)
=f_{H}^{-1}\left( c+te^{i\theta }\right) $, which means that $\phi
_{t}^{\theta }\left( f_{H}^{-1}\left( c\right) \right) \subset
f_{H}^{-1}\left( c+te^{i\theta }\right) $. The opposite inclusion is
obtained applying the inverse flow $\phi _{-t}^{\theta }$, and we conclude
that $\phi _{t}^{\theta }\left( f_{H}^{-1}\left( c\right) \right)
=f_{H}^{-1}\left( c+te^{i\theta }\right) $. Thus, $\phi _{t}^{\theta }$ is a
diffeomorphism between $f_{H}^{-1}\left( c\right) =f_{H}^{-1}\left(
c+te^{i\theta }\right) $.

This shows that every regular value in the open ball $B\left( c,\sigma
_{\varepsilon }\right) $ is equivalent to $c$, that is, its fibre is
diffeomorphic to the fibre at $c$.
\end{proof}

\subsection{Symplectic form}

\label{sec-symp}

The symplectic form that solves item (3) of theorem \ref{teofibracao} is the
imaginary part of the Hermitian form $\mathcal{H}$ from (\ref{hermitian}).
We write the real and imaginary parts of $\mathcal{H}$ as 
\begin{equation*}
\mathcal{H}\left( X,Y\right) =\left( X,Y\right) +i\Omega \left( X,Y\right)
\qquad X,Y\in \mathfrak{g}. 
\end{equation*}%
The real part $\left( \cdot ,\cdot \right) $ is an inner product (since $%
\left( X,X\right) =\mathcal{H}\left( X,X\right) $) and the imaginary part of 
$\Omega $ is a symplectic form on $\mathfrak{g}$. Indeed, we have 
\begin{equation*}
0\neq i\mathcal{H}\left( X,X\right) =\mathcal{H}\left( iX,X\right) =i\Omega
\left( iX,X\right) , 
\end{equation*}%
that is, $\Omega \left( iX,X\right) \neq 0$ for all $X\in \mathfrak{g}$,
which shows that $\Omega $ is nondegenerate. Moreover, $d\Omega =0$ because $%
\Omega $ is a constant bilinear form.

The fact that $\Omega \left( iX,X\right) \neq 0$ for all $X\in \mathfrak{g} $
guarantees that the restriction of $\Omega $ to any complex subspace of $%
\mathfrak{g}$ is also nondegenerate.

Now, the tangent spaces to $\mathcal{O}\left( H_{0}\right) $ are complex
vector subspaces of $\mathfrak{g}$. Therefore, the pullback of $\Omega $ by
the inclusion $\mathcal{O}\left( H_{0}\right) \hookrightarrow \mathfrak{g}$
defines a symplectic form on  $\mathcal{O}\left( H_{0}\right) $.

Finally, the subspaces tangent to the level manifolds $f_{H}^{-1}\left(
c\right) $ are complex subspaces of  $\mathfrak{g}$ as well. Thus, if $c$ is
a regular value then $f_{H}^{-1}\left( c\right) $ is a symplectic
submanifold of $\mathcal{O}\left( H_{0}\right) $.

This concludes the proof of item (3) of theorem \ref{teofibracao}.

\begin{remark}
An adjoint orbit $\mathcal{O}\left( X\right) \subset \mathfrak{g}$ admits
another natural symplectic form $\omega $ besides the form $\Omega $ defined
by $\mathcal{H}$. In fact, since $\mathfrak{g}$ is semisimple, the adjoint
representation is isomorphic to the co-adjoint representation (via the
Cartan--Killing form $\langle \cdot ,\cdot \rangle $). Hence, the general
construction of symplectic forms on co-adjoint orbits of
Kirillov--Kostant--Souriaux can be carried through to the adjoint orbits of $%
\mathfrak{g}$. This yields the symplectic form $\omega $ on $\mathcal{O}%
\left( X\right) $ defined by $\omega _{x}\left( [x,A],[x,B]\right) =\langle
x,[A,B]\rangle $, where $x\in \mathcal{O}\left( X\right) $ and $A,B\in 
\mathfrak{g}$ (recall that $[x,A],[x,B]\in T_{x}\mathcal{O}\left( X\right) $%
). Nevertheless, the regular fibres $f_{H}^{-1}\left( c\right) $ of $f_{H}$
are not symplectic submanifolds with respect to this $\omega $. In fact, the
vector $[x,H]$ is a tangent to $f_{H}^{-1}\left( c\right) $, since if $x\in
f_{H}^{-1}\left( c\right) $, then 
\begin{equation*}
\left( df_{H}\right) _{x}\left( [x,H]\right) =\langle H,[x,H]\rangle
=\langle \lbrack H,H],x\rangle =0. 
\end{equation*}%
If $x$ is a regular point, then $[x,H]\neq 0$, but if $[x,A]$ (with $x\in 
\mathcal{O}\left( X\right) $ and $A\in \mathfrak{g}$) is tangent to $%
f_{H}^{-1}\left( c\right) $ then 
\begin{equation*}
\omega _{x}\left( [x,H],[x,A]\right) =\langle x,[H,A]\rangle =0 
\end{equation*}%
since $0=\left( df_{H}\right) _{x}\left( [x,A]\right) =\langle
H,[A,x]\rangle =\langle x,[H,A]\rangle $.
\end{remark}

Now a few comments about the singular fibres. First a note on the special
case when $H_{0}\in \mathfrak{h}_{\mathbb{R}}$. Let $wH_{0}$, $w\in \mathcal{%
W}$, be a singularity. Define 
\begin{equation*}
\Pi \left( wH_{0}\right) =\{\alpha \in \Pi :\alpha \left( H_{0}\right) >0\}. 
\end{equation*}%
Then the subspaces 
\begin{equation*}
\mathfrak{n}^{\pm }\left( wH_{0}\right) =\sum_{\alpha \in \pm \Pi \left(
wH_{0}\right) }\mathfrak{g}_{\alpha } 
\end{equation*}%
are the nilpotent subalgebras of $\mathfrak{g}$. Let $N^{\pm }\left(
wH_{0}\right) $ be the connected groups with Lie algebra $\mathfrak{n}^{\pm
}\left( wH_{0}\right) $. Then the following result holds true (see Helgason):

\begin{itemize}
\item The map $n\in N^{+}\left( wH_{0}\right) \mapsto \mathrm{Ad}\left(
n\right) \left( wH_{0}\right) -wH_{0}\in \mathfrak{n}^{+}\left(
wH_{0}\right) $ is a diffeomorphism. Similarly, there is such an isomorphism
between $N^{-}\left( wH_{0}\right) $ and $\mathfrak{n}^{-}\left(
wH_{0}\right) $.
\end{itemize}

In particular, this implies that for all $n\in N^{\pm }\left( wH_{0}\right) $%
, $\mathrm{Ad}\left( n\right) \left( wH_{0}\right) =wH_{0}+X$ with $X\in 
\mathfrak{n}^{\pm }$. Therefore, 
\begin{equation*}
f_{H}\left( \mathrm{Ad}\left( n\right) wH_{0}\right) =\langle
H,wH_{0}+X\rangle =\langle H,wH_{0}\rangle =f_{H}\left( wH_{0}\right) . 
\end{equation*}%
Consequently, the complex subspaces $\mathrm{Ad}\left( N^{\pm }\left(
wH_{0}\right) \right) \left( wH_{0}\right) =\left( wH_{0}\right) +\mathfrak{n%
}^{\pm }\left( wH_{0}\right) $ are contained in the singular fibre $%
f_{H}^{-1}\left( \langle H,wH_{0}\rangle \right) $. This will be enough for
us to analyse the singular fibre on the next example. For higher dimensions
the structure of the singular fibres turns out rather more intricate, we
will address this issue in the forthcoming paper \cite{GGS}.

\section{Topology of  regular fibres}

\label{reg} To describe the regular fibres of $f_H$ we use another
description of the adjoint orbit, namely we regard it as a vector bundle. In
fact, the adjoint orbit has various realizations (e.g. as a homogeneous
space, and as the cotangent bundle of a flag manifold). These various
realizations, as well as their symplectic geometry, are explored in detail
in \cite{GGS}. The realization of the orbit as a cotangent bundle appeared
earlier in \cite{ABB}.

To study the topology of the regular fibres, we first identify  the orbit $\mathcal{O}\left( H_{0}\right)$ 
with the cotangent bundle of a flag manifold. Here is a summary of  the construction.
Let  $G$ be a  semisimple Lie group with  Lie algebra  $\mathfrak{g}$ and Cartan subalgebra $ \mathfrak{h}$.
The adjoint orbit  of an element  $H_0\subset \mathfrak{h}$  can be identified with the homogeneous space $G/Z_{H_0}$, 
where $Z_{H_0}$ is the centraliser of  $H_0$ in $G$.
We also identify the adjoint orbit $\mathrm{Ad}\left( K\right) \cdot H_{0}$  of the maximal compact subgroup  $K$ of $G$
 with the flag manifold $\mathbb{F}_{H_0}=G/P_{H_0}$, 
 where $P_{H_0}$ is the parabolic subgroup  which contains $Z_{H_0}$.
Using the construction of the vector bundle associated to the  $P_{H_0}$-principal bundle $G\rightarrow \mathbb{F}_{H_0}=G/P_{H_0}$ 
we showed that the quotient  $G/Z_{H_0}$  has the structure of  a vector bundle over $\mathbb{F}_{H_0}$ 
isomorphic to the cotangent bundle $T^\ast\mathbb{F}_{H_0}$
 \cite[thm. 2.1]{GGS}.


We now use the identification of the orbit with the cotangent bundle of a
flag to describe the regular fibres of $f_H$. Our height function $%
f_{H}\left( x\right) =\langle H,x\rangle $, $x\in \mathcal{O}\left(
H_{0}\right) $, 
takes values in $\mathbb{C}$, whereas, by hypothesis, $H$ and $H_{0}$ are
real, that is, belong to $\mathfrak{h}_{\mathbb{R}}$, and $H$ is regular. We
showed in proposition \ref{singularities} that $f_H$ has a finite number of
singularities. These singular points belong to $\mathbb{F}_{H_0 }$,
regarded as the orbit of the compact group $U\cdot H_{0}$.

Since $H$ and $H_{0}$ are real, $f_{H}$ restricted to $\mathbb{F}_{H_0}$
takes real values. $H$ and $H_{0}$ can be chosen in 
{\it general position} such that $\langle H,wH_{0}\rangle
=\langle H,uH_{0}\rangle $ if and only if $w=u$, where $w,u\in \mathcal{W}$.
(The latter condition implies that the singular levels do not intersect.
Such general position may be obtained by fixing $H_{0}$ then varying $H$.)

In this section and the next, when we use the identification of the 
adjoint orbit with the cotangent bundle of a flag manifold, the word fibre 
appears in two senses: a fibre of the Lefschetz fibration $f_H$ which is 
topologically nontrivial, and a fibre of the cotangent bundle $T^* \mathbb{F}_{H_0}$
which is a vector space. To avoid confusion between the two meanings of fibre, we introduce the term level:

\begin{definition} We call $L\left( \xi \right) =f_{H}^{-1}\left( f_{H}\left( \xi \right) \right) $
the {\it level} of $f_{H}$ passing through $\xi \in \mathcal{O}\left(
H_{0}\right) $. If $L\left( \xi \right)$ contains a singularity of $f_H$ we call it a 
{\it singular level}, otherwise we call it a {\it regular level}.
\end{definition}

\begin{notation}

$\widetilde{X}$ denotes the vector field on $\mathbb{F}_{H_0} $ induced
by $X\in \mathfrak{g}$, defined as $\widetilde{X}\left( x\right) =\frac{d}{dt%
}e^{tX}x_{\left\vert t=0\right. }$.
\end{notation}

\begin{theorem}
\label{teonivelreg} A regular level $L\left( \xi \right) $ is an {affine
subbundle} of the cotangent bundle restricted to the complement of the
singular points $\mathbb{F}_{H_0} \setminus \mathcal{W}\cdot H_{0}$. More
precisely, a regular level $L\left( \xi \right) $ surjects over $\mathbb{F}_{H_0} \setminus \mathcal{W}\cdot H_{0 }$ and its intersection with
the cotangent  fibre $T_{x}^{\ast }\mathbb{F}_{H_0}$ 
is an affine subspace, whose underlying vector space is 
\begin{equation*}
V_{H}\left( x\right) =\{\mu \in T_{x}^{\ast }\mathbb{F}_{H_0}:\mu \left( 
\widetilde{H}\left( x\right) \right) =0\}.
\end{equation*}
\end{theorem}

Identifying $T^{\ast }\mathbb{F}_{H_0}$ with the tangent bundle $T \mathbb{F}_{H_0}$ via the Borel metric, the subspace $V_{H}\left(
x\right) $ becomes the subspace orthogonal to $\widetilde{H}\left( x\right) $%
, which is exactly the space tangent to the level $x$ of the function $f_{H}$
restricted to the flag.

The proof of theorem \ref{teonivelreg} is a rather immediate consequence of
the construction of the action of $G$ on $T^{\ast }\mathbb{F}_{H_0}$,
that identifies it with the adjoint orbit $\mathcal{O}\left( H_{0}\right) =%
\mathrm{Ad}\left( G\right) \cdot H_{0}$. It involves the following facts:

\begin{enumerate}
\item The \textit{real part} of $f_{H}$ is known. In fact, let $\mathfrak{g}%
^{R}$ be the realification of $\mathfrak{g}$ (which is also a semisimplesimple
Lie algebra). Denote by $\langle \cdot ,\cdot \rangle ^{R}$ the 
Cartan--Killing form of $\mathfrak{g}^{R}$. Then, $\langle \cdot ,\cdot
\rangle ^{R}=2\text{Re}\langle \cdot ,\cdot \rangle .$ Thus, $\left( \text{Re%
}f_{H}\right) \left( x\right) =1/2f^{R}\left( x\right) $ where $f^{R}\left(
x\right) =\langle H,x\rangle ^{R}$.

\item The Cartan decomposition of $\mathfrak{g}$ (or rather of $\mathfrak{g}%
^{R}$) is given by $\mathfrak{g}=\mathfrak{u}\oplus i\mathfrak{u}$ where $%
\mathfrak{u}$ is the real compact form of $\mathfrak{g}$ and $\mathfrak{s}=i%
\mathfrak{u}$. The group $U=\langle \exp \mathfrak{u}\rangle $ is compact.
The exponential is taken to any group $G$ with Lie algebra $\mathfrak{g}$.

\item Since $\mathfrak{u}$ is a real compact form, it follows that the
restriction of the Cartan-Killing form $\langle \cdot ,\cdot \rangle $ to $%
\mathfrak{u}$ is negative definite (and takes real values). Hence, the
restriction to $i\mathfrak{u}$ is positive definite. Moreover, if $X\in 
\mathfrak{u}$ and $Y\in i\mathfrak{u}$ then $\langle X,Y\rangle $ is purely
imaginary.

\item The intersection $\mathcal{O}\left( H_0 \right) \cap i\mathfrak{%
u}$ coincides with the flag $\mathbb{F}_{H_0}=\mathrm{Ad}\left( U\right)
H_{0}$. 

\item The restriciton of $f_{H}$ to $\mathcal{O}\left( H_{0}\right) \cap i%
\mathfrak{u}=\mathbb{F}_{H_0}$ is real, equal to $1/2f^{R}$.

\item The \textit{imaginary part} of $f_{H}$ comes from $f_{iH}\left(
x\right) =\langle iH,x\rangle $, $x\in \mathcal{O}\left( H_{0}\right) $, in
the following way: 
\begin{equation*}
f_{iH}\left( x\right) =i\langle H,x\rangle =if_{H}\left( x\right) =-\text{Im}%
f_{H}\left( x\right) +i\text{Re}f_{H}\left( x\right) ,
\end{equation*}%
therefore $\text{Im}f_{H}\left( x\right) =-\text{Re}f_{iH}\left( x\right) =-%
\text{Re}\langle iH,x\rangle =-\frac{1}{2}\langle iH,x\rangle ^{R}$. Hence,%
\begin{equation*}
f_{H}=f_{H}^{R}-if_{iH}^{R}
\end{equation*}%
where the upper index indicates that the height function is taken with
respect to the real Cartan--Killing form $\langle \cdot ,\cdot \rangle ^{R}=2%
\text{Re}\langle \cdot ,\cdot \rangle $. This seemingly trivial formula is
 useful to express $f_{H}$ when we regard $\mathcal{O}\left( H_{0}\right) $
as $T^{\ast }\mathbb{F}_{H_0} $.

\item \textit{Height function on the cotangent bundle (real part):} 
If $X\in \mathfrak{s}=i\mathfrak{u}$ then $\alpha \left( X\right)
=X^{\#}+V_{X}$. This means that the vector field $\overrightarrow{X}$
induced by $X$ on $\mathcal{O}\left( H_{0}\right) $ is the Hamiltonian
vector field of the function $\left( \widetilde{X},\cdot \right)
_{B}+F_{X}^{R}$ where $\left( \cdot ,\cdot \right) _{B}$ is the Borel metric
on $\mathbb{F}_{H_0}$, $F_{X}^{R}=f_{X}^{R}\circ \pi $ and $\widetilde{X}
$ is the vector field induced by $X$ on $\mathbb{F}_{H_0}$.

In particular, the hypothesis that $H$ is real implies that $H\in \mathfrak{s%
}=i\mathfrak{u}$ and therefore the vector field $\overrightarrow{H}$ induced
by $H$ on $\mathcal{O}\left( H_{0}\right) $ is the Hamiltonian of the
function $\left( \widetilde{H},\cdot \right) _{B}+F_{H}^{R}$.
On the other hand, we know that the vector field $\overrightarrow{H}$ (given
by $\overrightarrow{H}\left( x\right) =[H,x]$) is the Hamiltonian of the
function $f_{H}^{R}\left( x\right) =\langle H,x\rangle ^{R}$ defined on $%
\mathcal{O}\left( H_{0}\right) $. Thus, the two functions give rise to the
same Hamiltonian fields and consequently differ by a constant. That is, via
the diffeomorphism between $\mathcal{O}\left( H_{0}\right) $ and $T^{\ast }%
\mathbb{F}_{H_0}$ the function $f_{H}^{R}\left( x\right) =\langle
x,H\rangle $ is given by $f_{H}^{R}=\left( \widetilde{H},\cdot \right)
_{B}+F_{H}^{R}+\mathrm{ct}$.\label{foraltcotan} 

\item \textit{Height function on the cotangent bundle (imaginary part):} the
imaginary part is given by $f_{iH}^{R}$. The difference here is that $iH\in 
\mathfrak{u}$, therefore $\overrightarrow{iH}$ is the Hamiltonian field of
the function $\left( \widetilde{iH},\cdot \right) _{B}$. But $%
\overrightarrow{iH}$ is the Hamiltonian field of $f_{iH}^{R}$ as well, thus $%
f_{iH}^{R}=\left( \widetilde{iH},\cdot \right) _{B}+\mathrm{ct}$. Together
with the previous item, this gives 
\begin{equation*}
f_{H}=\left( \widetilde{H},\cdot \right) _{B}+F_{H}^{R}-i\left( \widetilde{iH%
},\cdot \right) _{B}+\mathrm{ct}.
\end{equation*}

\item The constant of the previous item is calculated evaluating the
equality on $H_{0}$; terms involving the Borel metric vanish (zero section).
Therefore 
\begin{equation*}
\mathrm{ct}=f_{H}\left( H_{0}\right) -F_{H}^{R}\left( H_{0}\right)
=f_{H}\left( H_{0}\right) -f_{H}^{R}\left( H_{0}\right) =\langle
H,H_{0}\rangle -\langle H,H_{0}\rangle ^{R}=0
\end{equation*}%
since $\langle H,H_{0}\rangle $ is real.
\end{enumerate}

\vspace{12pt}

\noindent\textbf{Proof of theorem} \ref{teonivelreg}: Choose a regular point
$x\in \mathbb{F}_{H_0}=\mathcal{O}\left( H_{0}\right) \cap i\mathfrak{u}$%
. Then, the restriction of $f_{H}$ to the tangent space $T_{x}\mathbb{F}_{H_0}$ 
(identified with $T_{x}^{\ast }\mathbb{F}_{H_0}$ by the
Borel metric) is given by 
\begin{equation*}
\left( \widetilde{H}\left( x\right) ,\cdot \right) _{B}-i\left( \widetilde{iH%
}\left( x\right) ,\cdot \right) _{B}+f_{H}^{R}\left( x\right)
\end{equation*}%
which is an affine map, hence surjective. So, if  $x\in \mathbb{F}_{H_0}$
is a regular point of $f_H$  (that is, $x\in $ $\mathbb{F}_{H_0}\setminus \mathcal{%
W}\cdot H_{0}$) then every level of $f_{H}$ intercepts $T_{x}\mathbb{F}_{H_0}$. 
This shows that every regular level  $L( \xi ) $ projects
surjectivelly onto $\mathbb{F}_{H_0}\setminus \mathcal{W}\cdot H_{0}$. On
the other hand, the intersection of a level $L\left( \xi \right) $ with the
tangent space $T_{x}\mathbb{F}_{H_0}$ is given by the codimension $2$
affine subspace 
\begin{equation*}
L\left( \xi \right) \cap T_{x}\mathbb{F}_{H_0}=\{v\in T_{x}\mathbb{F}_{H_{0}}:\left( \widetilde{H}\left( x\right) ,v\right) _{B}-i\left( 
\widetilde{iH}\left( x\right) ,v\right) _{B}=f_{H}^{R}\left( x\right)
+f_{H}\left( \xi \right) \}
\end{equation*}%
which shows that $L\left( \xi \right) $ is an affine subbundle of $T^{\ast }%
\mathbb{F}_{H_0}$. \hfill $\square $

As a consequence we identify the topology of a regular level $L\left( \xi
\right) $:

\begin{corollary}
\label{cor.homology} The homology of a regular level $L\left( \xi \right) $
coincides with that of \,  $\mathbb{F}_{H_0}\setminus \mathcal{W}\cdot
H_{0}$. In particular, the middle  Betti number of $L\left( \xi \right)$
equals $k-1$, where $k$ is the number of singularities of the  fibration $f_H
$ (and equals the number of elements in the orbit $\mathcal{W}\cdot H_{0}$).
\end{corollary}

\section{Topology of  singular fibres}

\label{sing}

The singular levels of $f_{H}$ are the levels that pass through $wH_{0}$, $%
w\in \mathcal{W}$. Assume that $H_{0}$ and $H$ are in \textquotedblleft
general position\textquotedblright , so that each singular fibre contains
just one singularity.

The following proposition gives a description of the singular levels of $f_H$. In the
statement,\ $\pi \colon \mathcal{O}\left( H_{0}\right) \rightarrow \mathbb{F}_{H_0}$ 
is the canonical projection that makes $\mathcal{O}\left(
H_{0}\right) \approx T^{\ast }\mathbb{F}_{H_0}$, where $T^{\ast }\mathbb{F}_{H_0}$ is the flag manifold defined by $H_{0}$.

\begin{proposition}
\label{propfibrasing} The singular fibre of $f_{H}^{-1}\left( f_{H}\left(
wH_{0}\right) \right) $ passing through $wH_{0}$ is the disjoint union of
the following sets:

\begin{enumerate}
\item An affine subbundle of real codimension 2 of $\mathcal{O}\left(
H_{0}\right) \rightarrow \mathbb{F}_{H_0}\setminus \{uH_{0}:u\in \mathcal{%
W\}}$ over the set of regular points of \,\, $\mathbb{F}_{H_0}$.

\item The fibre 
$\pi^{-1}(wH_{0})$. As a subset of $\mathfrak{g}$ (in the adjoint
orbit) this fibre is given by the affine subspace 
\begin{equation*}
wH_{0}+\mathfrak{n}^{+}\left( wH_{0}\right)
\end{equation*}%
where $\mathfrak{n}_{w}^{+}$ is the sum of eigenspaces with positive
eigenvalues of $\mathrm{ad}\left( wH_{0}\right) $.
\end{enumerate}
\end{proposition}

The subspace $\mathfrak{n}^{+}\left( wH_{0}\right) $ in the statement is a
nilpotent subalgebra given by 
\begin{equation*}
\mathfrak{n}^{+}\left( wH_{0}\right) =\sum_{\alpha \in \Pi \left(
wH_{0}\right) }\mathfrak{g}_{\alpha }
\end{equation*}%
where $\Pi \left( wH_{0}\right) =\{\alpha \in \Pi :\alpha \left(
H_{0}\right) >0\}$.

\begin{proof}
To prove the proposition we examine the intersection of the level  $%
f_{H}^{-1}\left( f_{H}\left( wH_{0}\right) \right) $ with the fibres of $\pi
\colon\mathcal{O}\left( H_{0}\right) \rightarrow \mathbb{F}_{H_0} $.
Such intersections can be
described as follows:

\begin{enumerate}
\item Let $x\in \mathbb{F}_{H_0}$ be a regular point of $f_{H}$, that is, $%
x\neq uH_{0}$ for all $u\in \mathcal{W}$. Then, the restriction of  $f_{H}$
to the cotangent fibre $\pi ^{-1}\{x\}$ is an affine map, whose linear part
is nonzero. Such linear part is the functional $\left( \widetilde{H},\cdot
\right) _{B}-i\left( \widetilde{iH},\cdot \right) _{B}$, where $\left( \cdot
,\cdot \right) _{B}$ is the Borel metric). If $x\in \mathbb{F}_{H_0}$ is a
regular point, then the linear part has no zeros. This implies that all
levels of $f_{H}$ intersect $\pi ^{-1}\{x\}=T_{x}^{\ast }\mathbb{F}_{H_0}$
on affine subspaces of complex codimension $1$, proving statement (1).

\item Let $N^{+}\left( wH_{0}\right) $ be the connected group with Lie
algebra  $\mathfrak{n}^{+}\left( wH_{0}\right) $. Then, the map 
\begin{equation*}
n\in N^{+}\left( wH_{0}\right) \mapsto \mathrm{Ad}\left( n\right) \left(
wH_{0}\right) -wH_{0}\in \mathfrak{n}^{+}\left( wH_{0}\right)
\end{equation*}%
is a diffeomorphism. In particular, for all $n\in N^{+}\left( wH_{0}\right) $%
, $\mathrm{Ad}\left( n\right) \left( wH_{0}\right) =wH_{0}+X$ with $X\in 
\mathfrak{n}^{+}$. Therefore, 
\begin{equation}
f_{H}\left( \mathrm{Ad}\left( n\right) wH_{0}\right) =\langle
H,wH_{0}+X\rangle =\langle H,wH_{0}\rangle =f_{H}\left( wH_{0}\right) .
\label{forfaga}
\end{equation}%
Hence, the affine subspace $\mathrm{Ad}\left( N^{+}\left( wH_{0}\right)
\right) \left( wH_{0}\right) =\left( wH_{0}\right) +\mathfrak{n}^{+}\left(
wH_{0}\right) $ is contained in the singular level $f_{H}^{-1}\left( \langle
H,wH_{0}\rangle \right) $.

Using the isomorphism $\mathcal{O}\left( H_{0}\right) \approx T^{\ast }%
\mathbb{F}_{H_{0}}$, we see that the fibre over $wH_{0}$ is precisely $%
\left( wH_{0}\right) +\mathfrak{n}^{+}\left( wH_{0}\right) $, proving
statement (2).

\item It remains to verify that if $uH_{0}\neq wH_{0}$ then the fibre $\pi
^{-1}\{uH_{0}\}$ does not intersect the level $f_{H}^{-1}\left( \langle
H,wH_{0}\rangle \right) $. By the same argument as in the previous item, the
fibre $\pi ^{-1}\{uH_{0}\}$ in the adjoint orbit, is given by the adjoint
subspace  $\left( uH_{0}\right) +\mathfrak{n}^{+}\left( uH_{0}\right) $. By
equalities (\ref{forfaga}) $f_{H}$ is constant on this subspace and equals  $%
f_{H}\left( uH_{0}\right) $. Since by hypothesis each singular level
contains just one singularity, this shows that $f_{H}^{-1}\left( \langle
H,wH_{0}\rangle \right) $ does not intersect the fibre over $uH_{0}\neq
wH_{0}$.
\end{enumerate}
\end{proof}

\begin{corollary}\label{cor.sing}
The homology of a singular level $L\left( w H_0\right) $, $w \in \mathcal{W}$ 
coincides with that of \,  
$$\mathbb{F}_{H_0} \setminus  \{uH_0 \in \mathcal{W}\cdot H_{0}: u\neq w\}\text{.}$$
 In particular, the middle  Betti number of $L\left( wH_0 \right)$
equals $k-2$, where $k$ is the number of singularities of the  fibration $f_H.$
\end{corollary}

\begin{example}
In the case of $\mathrm{Sl}\left( 2,\mathbb{C}\right) $ the singular fibres
are just the union of 2 subspaces. In this case the affine bundle has rank 0
and each fibre of this bundle intersects  $H_{0}+\mathfrak{n}^{-}\left(
H_{0}\right) $ as well as $\left( w_{0}H_{0}\right) +\mathfrak{n}^{-}\left(
w_{0}H_{0}\right) $ with $w_{0}H_{0}=-H_{0}$. We conclude that this subbundle
is contained in the affine spaces $H_{0}+\mathfrak{n}^{-}\left( H_{0}\right) 
$ and $\left( w_{0}H_{0}\right) +\mathfrak{n}^{-}\left( w_{0}H_{0}\right) $
which are part of the singular levels of $H_{0}$ and $w_{0}H_{0}=-H_{0}$,
respectively.
\end{example}

\end{document}